\documentclass{article}

\usepackage{amsmath,amssymb,amsthm}

\usepackage[T1]{fontenc}
\usepackage{kpfonts, baskervald}
\newtheorem{thm}{Theorem}
\numberwithin{thm}{section}
\newtheorem{cor}[thm]{Corollary}
\newtheorem{prop}[thm]{Proposition}

\theoremstyle{definition}
\newtheorem{defn}[thm]{Definition}
\theoremstyle{remark}
\newtheorem{rmk}[thm]{Remark}

\usepackage{microtype}

\usepackage{tikz-cd}

\usepackage{todonotes}

\usepackage{hyperref}
\hypersetup{
  colorlinks   = true, 
  urlcolor     = blue, 
  linkcolor    = blue, 
  citecolor   = red 
}

\usepackage{epigraph}
\newcommand{\co}{\colon\thinspace}
\newcommand{\mb}[1]{\mathbb{#1}}
\newcommand{\mf}[1]{\mathfrak{#1}}
\newcommand{\mc}[1]{\mathcal{#1}}
\newcommand{\mi}[1]{\mathscr{#1}}

\newcommand{\op}{{op}}

\newcommand{\too}{\xrightarrow}

\newcommand{\NFin}{{\mathrm{NFin}_*}}

\DeclareMathOperator{\Hom}{Hom}
\DeclareMathOperator{\Map}{Map}

\DeclareMathOperator{\Nat}{Nat}

\DeclareMathOperator{\Fun}{Fun}
\DeclareMathOperator{\Mul}{Mul}
\DeclareMathOperator{\Sec}{Sec}

\DeclareMathOperator{\ev}{ev}

\DeclareMathOperator{\Set}{Set}
\DeclareMathOperator{\sSet}{sSet}

\DeclareMathOperator{\Alg}{Alg}
\DeclareMathOperator{\Cat}{Cat}
\DeclareMathOperator{\Op}{Op}

\DeclareMathOperator{\Inert}{Inert}

\DeclareMathOperator*{\into}{\hookrightarrow}

\DeclareMathOperator*{\bvt}{\otimes}

\title{Lax monoidality for products of enriched higher categories}
\author{Tyler Lawson\thanks{The author was partially supported by NSF
    grant 2208062.}}

\begin{document}
\maketitle

\begin{abstract}
  We prove that a lax $\mathbb{E}_{n+1}$-monoidal functor from $\mathcal V$ to $\mathcal W$ induces a lax $\mathbb{E}_n$-monoidal functor from $\mathcal V$-enriched $\infty$-categories to $\mathcal W$-enriched $\infty$-categories in the sense of Gepner--Haugseng.

  We prove this as part of a general-purpose interaction with the Boardman--Vogt tensor product $\otimes$: given a construction that takes an $\mathcal E$-monoidal $\infty$-category to a category expressible in diagrammatic terms, we give a criterion for it to take $(\mathcal{O} \otimes \mathcal{E})$-monoidal $\infty$-categories to $\mathcal{O}$-monoidal $\infty$-categories using a ``pointwise'' monoidal structure.
\end{abstract}

\section{Introduction}

\begin{epigraphs}
\qitem{The cases where $n \leq p$ are comparatively easy. [...] We will, however, give a unified proof which makes these cases look just as bad as the others.}{J. W. Milnor, \textit{Topology from the differentiable viewpoint}}
\end{epigraphs}
\begin{epigraphs}
\qitem{Every time I look, the baby's eating sand. I turn around, sand. \\Where does this sand come from? \\I don't know. So, I let them eat it.}{\textit{Orphan Black}}
\end{epigraphs}

Classical enriched categories are functorial, and we can take products of them. More specifically:
\begin{itemize}
\item If $f\co V \to W$ is lax monoidal, then we can base-change along $f$, turning any $V$-enriched category $C$ into a W-enriched category $f_* C$.
\item If $C$ is $V$-enriched and $D$ is $W$-enriched, then there is an external product $C \boxtimes D$ which is a $(V \times W)$-enriched category: we take the product of the object sets and the Hom-objects.
\item These interact well: if $V$ is symmetric monoidal, then the product $m\co V \times V \to V$ is a lax monoidal functor, and the product $C \mathop\boxtimes^V D = m_*(C \boxtimes D)$ is part of a symmetric monoidal structure on $V$-enriched categories.
\item This symmetric monoidal structure is functorial: for a lax symmetric monoidal functor $f\co V \to W$, the functor $f_*\co \Cat^V \to \Cat^W$ is also lax symmetric monoidal.
\end{itemize}

Except for one detail, all of this has already been generalized to enriched $\infty$-categories. For example, Gepner and Haugseng proved that the external product makes the functor $\mc V \mapsto \Cat_\infty^{\mc V}$ lax symmetric monoidal from monoidal $\infty$-categories to $\infty$-categories \cite[5.7.11]{gepner-haugseng-enriched}. If $\mc V$ is $\mb E_{n+1}$-monoidal, then they applied this to give $\Cat_\infty^{\mc V}$ the structure of an $\mb E_{n}$-monoidal $\infty$-category, and gave generalizations using the Boardman--Vogt style tensor product of $\infty$-operads \cite[2.2.5]{lurie-higheralgebra}. Alternative approaches were also developed by Heine \cite[7.17]{heine-enriched} and Hinich \cite[3.5.3]{hinich-enrichedyoneda}.

Unfortunately, while $\Cat_\infty^{\mc V}$ is functorial in \emph{lax} monoidal functors \cite[5.7.6]{gepner-haugseng-enriched}, in the above references the $\mc O$-monoidal structures on $\Cat_\infty^{\mc V}$ are only functorial in \emph{strong} monoidal functors. For instance, Gepner--Haugseng's proof produces a lax symmetric monoidal functor
\[
  \Cat_\infty^{(-)}\co \Alg_{\mb E_1}(\Cat_\infty) \to \Cat_\infty.
\]
We can apply $\Alg_{\mb E_n}$ to both sides to get a functor
\[
  \Alg_{\mb E_{n+1}}(\Cat_\infty) \to \Alg_{\mb E_n}(\Cat_\infty)
\]
by Dunn additivity. However, for any $\infty$-operad $\mc O$, maps $\mc V \to \mc W$ in $\Alg_{\mc O}(\Cat_\infty)$ are strong $\mc O$-monoidal functors rather than lax ones. (Similarly, Heine's construction in terms of $\mc O$-algebras in monoidal categories, and Hinich's in terms of $\mc O$-algebras in planar operads, both unwind to functoriality in strict functors.)

The main application of this note is that we have this lax functoriality. For the majority of readers, the following result is probably as much detail as is needed.

\begin{thm}
  \label{thm:mainthm}
  Let $\Cat_\infty^{(-)}$ denote the functor $\mc V \mapsto \Cat_\infty^{\mc V}$, sending an $\mb E_1$-monoidal $\infty$-category to the category of $\mc V$-enriched $\infty$-categories. Then $\Cat^{(-)}$ extends to a functor from $\mb E_{n+1}$-monoidal $\infty$-categories and \emph{lax} functors to $\mb E_n$-monoidal $\infty$-categories and lax functors.
\end{thm}

A few words about the approach are in order. Given any $\mc O$-monoidal $\infty$-categories $\mc C$ and $\mc D$, expressed as coCartesian fibrations $\mc C^\otimes \to \mc O^\otimes$ and $\mc D^\otimes \to \mc O^\otimes$, a lax $\mc O$-monoidal functor is expediently described as a map $\mc C^\otimes \to \mc D^\otimes$ of $\infty$-operads over $\mc O^\otimes$. The proof of this theorem will therefore follow by exhibiting $\Cat_\infty^{(-)}$ as induced by a right adjoint in a Quillen adjunction, between marked simplicial sets over $\mc O^\otimes$ and marked simplicial sets over the Boardman--Vogt tensor product $(\mc O \bvt \mb E_1)^\otimes$.

In an ideal future, the proof in this paper will become obsolete. The natural home for a result of this type should probably be a systematic study of  compatibility between the Boardman--Vogt tensor product of $\infty$-operads and lax structures. This paper is the result of losing this particular game of ``hot potato'': Theorem~\ref{thm:mainthm} is important in an application, but the author is not currently equipped to develop an $(\infty,2)$-categorical framework. We look forward to having a good laugh about it someday.

Of necessity, this paper is of a rather technical nature. Our proof that $\Cat_\infty^{(-)}$ is part of a Quillen adjunction uses the formalism of categorical patterns developed in \cite[Appendix B]{lurie-higheralgebra}. Many of the components of the proof are upgraded versions of the construction of the $\infty$-operad $\Fun(\mc C, \mc D)^\otimes$ of functors from one $\infty$-operad to another in \cite[2.2.6]{lurie-higheralgebra}, used there to set up the Day convolution. In the case at hand, our proof uses the following setup for enriched $\infty$-categories.
\begin{itemize}
\item We have a marked simplicial set $\Delta^\op$ with a fixed map $\Delta^\op \to \mb E_1^\otimes$ to the associative $\infty$-operad.
\item For any set $X$, we have a marked simplicial set $\Delta^\op_X$ with functor $\Delta^\op_X \to \Delta^\op$ \cite[4.1.1]{gepner-haugseng-enriched}; the fiber over $[n]$ is $X^{n+1}$.
\item The construction $X \mapsto \Delta^\op_X$ is functorial in $X$, and is symmetric monoidal: it takes products of sets to fiber products over $\Delta^\op$.
\item For a map $\mc V^\otimes \to \mb E_1^\otimes$ of $\infty$-operads, a $\mc V$-enriched $\infty$-category is defined to be a pair $(X,s)$ of a set $X$ of ``objects'' and a functor $s\co \Delta^\op_X \to \mc V^\otimes$ giving a commutative diagram of marked simplicial sets:
  \[
    \begin{tikzcd}
      \Delta^\op_X \ar[r,"s",dashed] \ar[d] & \mc V^\otimes \ar[d] \\
      \Delta^\op \ar[r] & \mb E_1^\otimes
    \end{tikzcd}
  \]
  For any $(a,b) \in X \times X$, the image of $(a,b)$ in $\mc V$ represents an enriched $\Hom$-object $\Hom^{\mc V}(a,b)$ of maps from $a$ to $b$. Similarly, maps are appropriately compatible pairs of a map of sets and a natural transformation of functors.
\end{itemize}
For any $\infty$-operad $\mc O^\otimes$, we will upgrade this construction to a functor
\[
  \Cat_\infty^{(-)}\co \Op_\infty / (\mc O \bvt \mb E_1)^\otimes \to \Op_\infty / \mc O^\otimes
\]
which allows us to conclude that $\Cat_\infty^{(-)}$ takes lax $(\mc O \bvt \mb E_1)$-monoidal functors to lax $\mc O$-monoidal functors (see Theorem~\ref{thm:thmmainmain}).

To help these results find their natural generality, we have given a general-purpose result (Theorem~\ref{thm:mainspan}) which produces a ``pointwise'' monoidal structure on certain structures in monoidal $\infty$-categories. This allows us to obtain similar results whenever we have similar identities to those satisfied by the map $\Delta^\op \to \mb E_1^\otimes$ and the functor $\Delta^\op_{(-)}$.

The author would like to thank
Clark Barwick,
Rune Haugseng,
and
Shay Ben Moshe
for helpful conversations related to this material.

\section{Inert maps}

\begin{defn} 
 \label{def:inertcategory}
  Fix an $\infty$-operad $\mc A^\otimes$, and let $\Inert(\mc A^\otimes) \subset \Fun(\Delta^1, \mc A^\otimes)$ be the full sub-$\infty$-category spanned by the inert morphisms.
\end{defn}

\begin{rmk}
  \label{rmk:inertclassifier}
  Evaluation at $1 \in \Delta^1$ determines a functor $\ev_1\co \Inert(\mc A^\otimes) \to \mc A^\otimes$. This is a coCartesian fibration, pulled back from $\Inert(\NFin) \to \NFin$, and classifies a composite functor
  \[
    \mc A^\otimes \to \NFin \too{\mi P} \Set.
  \]
  Here the functor $\mi P$ sends a finite pointed set $S_+$ to the power set $\mi P(S)$, and sends a pointed map $S_+ \to T_+$ to the ``non-basepoint image'' functor $\mi P(S) \to \mi P(T)$. This functor has a right adjoint: the ``inverse image'' functor $\mi P(T) \to \mi P(S)$, and so the functor $\ev_1\co \Inert(\mc A^\otimes) \to \mc A^\otimes$ is also a Cartesian fibration.

  Evaluation at $0$ determines a second functor $\ev_0\co \Inert(\mc A^\otimes) \to \mc A^\otimes$.
\end{rmk}

\begin{defn}
  \label{def:expansion}
  Suppose that $\mc A^\otimes$ is an $\infty$-operad and $p\co \Gamma \to \mc A^\otimes$ is a map of simplicial sets.
  We define the \emph{inert expansion of $\Gamma$} to be the simplicial set
  \[
    \mc K^\Gamma = \mc K \times_{\mc A^\otimes} \Inert(\mc A),
  \]
  formed as the pullback in the following diagram:
  \[
    \begin{tikzcd}
      \mc K^\Gamma \ar[r,"e"] \ar[d,"\pi",swap] & \Gamma \ar[d,"p"] \\
      \Inert(\mc A) \ar[r,"\ev_1",swap] & \mc A^\otimes
    \end{tikzcd}
  \]
  (If $\Gamma$ is understood, we will denote it simply by $\mc K$.)
  The inert expansion has two maps $\mc K^\Gamma \to \mc A^\otimes$: the map $\ev_1 \circ \pi$ in the above diagram, and the map $s = \ev_0 \circ \pi$ which we refer to as the \emph{source} map.
\end{defn}

\begin{prop}
  \label{prop:expansioncocartesian}
  Suppose $p\co \Gamma \to \mc A^\otimes$ is a coCartesian fibration. Then so is the map $s\co \mc K^\Gamma \to \mc A^\otimes$. In particular, $\mc K^\Gamma$ is an $\infty$-category.

  Under these circumstances, a map $f$ in $\mc K^\Gamma$ is $s$-coCartesian if and only if $\pi(f)$ is $(\ev_0)$-coCartesian in $\Inert(\mc A)$ and $e(f)$ is $p$-coCartesian in $\Gamma$.
\end{prop}

\begin{proof}
  CoCartesian fibrations are stable under pullback, and hence $\pi\co \mc K^\Gamma \to \Inert(A)$ is a coCartesian fibration. By \cite[2.4.2.3]{lurie-htt}, the composite $s\co \mc K \to \Inert(A) \to A^\otimes$ is also a coCartesian fibration with the stated coCartesian morphisms.
\end{proof}

\begin{cor}
  \label{cor:expansionflat}
  If the map $p$ is a coCartesian fibration, then the map $s$ is a flat categorical fibration.
\end{cor}

\begin{proof}
  CoCartesian fibrations are flat by \cite[B.3.11]{lurie-higheralgebra}.
\end{proof}

\begin{prop}
  \label{prop:expansioncartesian}
  Suppose $p\co \Gamma \to \mc A^\otimes$ is a coCartesian fibration and $\gamma\co \Delta^1 \to \mc A^\otimes$ is an inert morphism. If we define $\mc K^\Gamma|_\gamma = \mc K^\Gamma \times_{\mc A^\otimes} \Delta^1$, then the restriction
  \[
    s|_\gamma\co \mc K^\Gamma|_\gamma \to \Delta^1
  \]
  is a Cartesian fibration.

  Under these circumstances, a map $f$ in $\mc K^\Gamma|_\gamma$ is $(s|_\gamma)$-Cartesian if and only if both $\ev_1\pi (f)$ and $e(f)$ are equivalences.
\end{prop}

\begin{proof}
  Suppose that for $i \in \{1,2\}$ we have inert maps $g_i\co U_i \to \overline U_i$ in $\mc A^\otimes$ and $Z_i \in \Gamma_{\overline U_i}$. Then the space of maps $(g_1, Z_1) \to (g_2, Z_2)$ in $\mc K^\Gamma$ is computed as a natural iterated (homotopy) pullback:
  \begin{equation}
    \label{eq:cartesianpullback}
    \begin{tikzcd}
      \Map_{\mc K^\Gamma}((g_1,Z_1), (g_2,Z_2)) \ar[r] \ar[d] &
      \Map_{\Inert(\mc A)}(g_1, g_2) \ar[r] \ar[d] &
      \Map_{\mc A^\otimes}(U_1, U_2) \ar[d] \\
      \Map_{\Gamma} (Z_1, Z_2) \ar[r] &
      \Map_{\mc A^\otimes}(\overline U_1, \overline U_2) \ar[r] &
      \Map_{\mc A^\otimes}(U_1, \overline U_2)
    \end{tikzcd}
  \end{equation}
  Now take any inert map $\alpha\co X \to Y$ in $\mc A^\otimes$ together with a lift of $Y$ to $\mc K^\Gamma$ along $s$, in the form of a pair $(h, W)$ of an inert morphism $h\co Y \to \overline Y$ and an object $W$ in the fiber $\Gamma_{\overline Y} = p^{-1} \overline Y$. The composite $h\alpha\co X \to Y \to \overline Y$ is also inert, and so $(h\alpha, W)$ is a lift of $X$ to $\mc K^\Gamma$ with a map $\tilde \alpha\co (h\alpha,W) \to (h,W)$ determined by the commutative diagram
  \[
    \begin{tikzcd}
      X \ar[r,"\gamma"] \ar[d,"h\gamma",swap] & Y \ar[d,"h"] \\
      \overline Y \ar[r,equal] & \overline Y
    \end{tikzcd}
  \]
  in $\Inert(\mc A)$. We wish to show that this map $\tilde \alpha$ is 
$s$-Cartesian.
  
  For any $(g,Z)$ in $\mc K^\Gamma$ consisting of an inert morphism $g\co U \to \overline U$ in $\mc A^\otimes$ and $Z \in \Gamma_{\overline U}$, we consider the following diagram:
  \[
    \begin{tikzcd}
      \Map_\Gamma(Z, W) \ar[r] \ar[d] &
      \Map_{\mc A^\otimes}(U, \overline Y) \ar[d] &
      \Map_{\mc A^\otimes}(U, X) \ar[d] \ar[l]\\
      \Map_\Gamma(Z, W) \ar[r] &
      \Map_{\mc A^\otimes}(U, \overline Y) &
      \Map_{\mc A^\otimes}(U, Y) \ar[l]
    \end{tikzcd}
  \]
  The vertical maps in the left-hand square are equivalences, so the left-hand square is a homotopy pullback square. Therefore, applying Equation~\eqref{eq:cartesianpullback} to identify the homotopy pullbacks in each row, we find that we get a homotopy pullback square
  \[
    \begin{tikzcd}
      \Map_{\mc K^\Gamma}((g,Z), (h\alpha, W)) \ar[r] \ar[d] &
      \Map_{\mc A^\otimes}(U, X) \ar[d]\\
      \Map_{\mc K^\Gamma}((g,Z), (h, W)) \ar[r] &
      \Map_{\mc A^\otimes}(U, Y)
    \end{tikzcd}
  \]
  Because $(g,Z)$ was arbitrary, this shows that the map $(h\alpha, W) \to (h,W)$ is $s$-Cartesian by \cite[2.4.1.10]{lurie-htt}.

  Given any inert edge $\gamma\co \Delta^1 \to \mc A^\otimes$, any map $\alpha$ in $\Delta^1$ has inert image in $\mc A^\otimes$. Then $\alpha$ has $s$-Cartesian lifts as above, which lie in $\mc K^\Gamma|_\gamma$; these lifts are automatically also $(s|_\gamma)$-Cartesian. Moreover, a general map $f$ is $(s|_\gamma)$-Cartesian if and only if it is equivalent to one of these, which is true if and only if $\ev_0(f)$ and $e(f)$ are equivalences.
\end{proof}

\section{Markings}

\begin{defn}
  \label{def:compatiblemarking}
  Suppose $p\co \Gamma \to \mc A^\otimes$ is a coCartesian fibration. A \emph{compatible marking} is a marking on the simplicial set $\Gamma$ such that:
  \begin{itemize}
  \item the map $p$ takes marked edges to inert morphisms,
  \item $p$-coCartesian lifts of inert morphisms are marked, 
  \item marked edges are closed under composition, and
  \item for any inert $\gamma\co U \to V$, the induced functor $\gamma_!\co \Gamma_U \to \Gamma_V$ preserves marked edges.
  \end{itemize}
\end{defn}

\begin{rmk}
  In particular, note that equivalences are automatically marked.
\end{rmk}

\begin{defn}
  \label{def:associatedmarking}
  Suppose $p\co \Gamma \to \mc A^\otimes$ is a coCartesian fibration with a compatible marking, and form the inert expansion as the pullback:
  \[
    \begin{tikzcd}
      \mc K^\Gamma \ar[r,"e"] \ar[d,"\pi",swap] & \Gamma \ar[d,"p"] \\
      \Inert(\mc A) \ar[r,"\ev_1",swap] & \mc A^\otimes
    \end{tikzcd}
  \]
The \emph{associated marking} on the inert expansion $\mc K^\Gamma$ is the set of edges $f$ of $\mc K^\Gamma$ such that:
  \begin{itemize}
  \item the image $s(f) = \ev_0 \pi(f)$ is inert in $\mc A^\otimes$, and
  \item the image $e(f)$ is marked in $\Gamma$.
  \end{itemize}
  This equips $\mc K^\Gamma$ with the structure of a marked simplicial set, and the map $s$ is a map of marked simplicial sets.
\end{defn}


\begin{prop}
  \label{prop:markedclosure}
  Suppose $p\co \Gamma \to \mc A^\otimes$ is a coCartesian fibration with a compatible marking. Then, in the associated marking on $\mc K^\Gamma$, marked edges are closed under equivalences and composition.
\end{prop}

\begin{proof}
  The maps $s$ and $e$ preserve composition and equivalences. Therefore, it suffices to observe that inert morphisms in the $\infty$-operad $\mc A^\otimes$ are closed under composition and equivalences by \cite[2.4.1.5, 2.4.1.7]{lurie-htt}, and that marked edges in $\Gamma$ are closed under composition and equivalences by assumption.
\end{proof}

\begin{prop}
  \label{prop:trianglecart}
  Let $p\co \Gamma \to \mc A^\otimes$ be a coCartesian fibration with a compatible marking. Suppose we are given a commutative diagram
  \[
    \begin{tikzcd}
      & Y \ar[dr, "g"]\\
      X \ar[ur, "f"] \ar[rr,"h"] && Z
    \end{tikzcd}
  \]
  in $\mc K^\Gamma$, where $g$ is locally $s$-Cartesian, $s(g)$ is inert, and $s(f)$ is an equivalence. Then $f$ is marked if and only if $h$ is marked.
\end{prop}

\begin{proof}
  The assumptions imply that $s(f)$ and $s(g)$ are inert, so $s(h)$ is inert as well. By Proposition~\ref{prop:expansioncartesian}, since $g$ is locally $s$-Cartesian, $e(g)$ is an equivalence in $\Gamma$. By compatibility of the marking on $\Gamma$, $e(f)$ is marked if and only if $e(h)$ is marked.
\end{proof}

\begin{prop}
  \label{prop:trianglecocart}
  Let $p\co \Gamma \to \mc A^\otimes$ be a coCartesian fibration with a compatible marking. Suppose we are given a commutative diagram
  \[
    \begin{tikzcd}
      & Y \ar[dr, "g"]\\
      X \ar[ur, "f"] \ar[rr,"h"] && Z
    \end{tikzcd}
  \]
  in $\mc K^\Gamma|_\gamma$ for some inert edge $\gamma$ in $\mc A^\otimes$, where $f$ is $(s|_\gamma)$-coCartesian and $(s|_\gamma)(g)$ is an equivalence. Then $g$ is marked if and only if $h$ is marked.
\end{prop}

\begin{proof}
  If $\gamma\co \Delta^1 \to \mc A^\otimes$ is inert, then all edges of $\Delta^1$ map to inert edges of $\mc A^\otimes$. Therefore, $s(f)$, $s(g)$, and $s(h)$ are all inert edges of $\mc A^\otimes$.

  By Proposition~\ref{prop:expansioncocartesian}, $e(f)$ is $p$-coCartesian, and hence marked in $\Gamma$ by compatibility of the marking. If $e(g)$ is marked, so is $e(h)$ because marked edges in $\Gamma$ are closed under composition.

  For the converse, suppose $e(h)$ is marked. Let $k\co \gamma^! Z \to Z$ be a locally $s$-Cartesian lift of $\gamma$ with associated diagram
  \[
    \begin{tikzcd}
      X \ar[r,"f"] \ar[d,"m",swap] \ar[dr,"h"] &
      Y \ar[d,"g"] \\
      \gamma^! Z \ar[r,"k",swap] &
      Z.
    \end{tikzcd}
  \]
  Applying Proposition~\ref{prop:trianglecart} to the lower-left triangle, we find that $m$ is marked, and hence so is $e(m)$. Therefore, applying $e$ to this diagram, we get a diagram
  \[
    \begin{tikzcd}
      e(X) \ar[r,"e(f)"] \ar[d,"e(m)",swap] &
      e(Y) \ar[d,"e(g)"] \\
      e(\gamma^! Z) \ar[r,"e(k)",swap] &
      e(Z)
    \end{tikzcd}
  \]
  in $\Gamma$. The map $e(f)$ is $p$-coCartesian, and $e(k)$ also $p$-coCartesian because it is an equivalence by Proposition~\ref{prop:expansioncartesian}. Therefore, this diagam is equivalent to a diagram of the form
  \[
    \begin{tikzcd}
      e(X) \ar[r] \ar[d,"e(m)",swap] &
      \gamma_! e(X) \ar[d,"\gamma_! e(m)"] \\
      e(\gamma^! Z) \ar[r] &
      \gamma_! e(\gamma^! Z).
    \end{tikzcd}
  \]
  By compatibility of the marking on $\Gamma$, since $e(m)$ is marked, so is $\gamma_! e(m) = e(g)$, as desired.
\end{proof}

\section{Spans}

\begin{rmk}
  \label{rmk:spanfunctor}
Suppose that we have a span of $\infty$-categories in the form of a diagram
\[
   \mc A \xleftarrow{p} \Gamma \xrightarrow{q} \mc C
\]
such that $p$ is a coCartesian fibration.

Given such a span, each object $U \in \mc A$ gives rise to a category $\Gamma_U$ with a functor $q|_U\co \Gamma_U \to \mc C$. Each map $\rho\co U \to V$ gives rise to a functor $\rho_!\co \Gamma_U \to \Gamma_V$, as well as a natural transformation $T_\rho\co \Delta^1 \times \Gamma_U \Rightarrow \mc C$ of functors from $q|_U$ to $q|_V \circ \rho_!$. Both of these are well-defined up to equivalence, and the latter satisfies $T_{\rho \sigma} \sim (T_\rho \circ \sigma_!) \cdot T_\sigma$.
\end{rmk}
\begin{defn}
  \label{def:sumpreserving}
  Suppose that $\mc A^\otimes$ and $\mc C^\otimes$ are $\infty$-operads and that we have a diagram
  \[
   \mc A^\otimes \xleftarrow{p} \Gamma \xrightarrow{q} \mc C^\otimes
  \]
  with $p$ a coCartesian fibration. We say that this diagram is \emph{sum-preserving} if, for any collection of maps
  \[
    \alpha_i\co U \to U_i
  \]
  exhibiting $U$ as the sum $\oplus U_i$ in the $\infty$-operad $\mc C^\otimes$ \cite[2.1.1.15]{lurie-higheralgebra}, the associated natural transformations
  \[
    T_{\alpha_i}\co q|_U \Rightarrow q|_{U_i} \circ (\alpha_i)_!
  \]
  of functors $\Gamma_{U} \to \mc C^\otimes$ exhibit $q|_{U}$ as a sum $\oplus q|_{U_i} \circ (\alpha_i)_!$ of functors to $\mc C^\otimes$.
\end{defn}

\begin{rmk}
  \label{rmk:sumterms}
  For example, this asks that for any sum diagram $U \xleftarrow{\alpha} U \oplus V \xrightarrow{\beta} V$ and any object $X \in \Gamma_{U\oplus V}$, the diagram
  \[
    q(\alpha_! X) \leftarrow q(X) \rightarrow q(\beta_! X)
  \]
  is a direct sum diagram in $\mc C^\otimes$. In particular, any coCartesian lift $X \to \rho_! X$ of an inert morphism must be sent by $q$ to an inert morphism.
\end{rmk}



\begin{defn}
  \label{def:spanexpansion}
  Suppose that $\mc A^\otimes$ and $\mc C^\otimes$ are $\infty$-operads and that we have a diagram
  \[
    \mc A^\otimes \xleftarrow{p} \Gamma \xrightarrow{q} \mc C^\otimes
  \]
  with $p$ a coCartesian fibration. Apply the inert expansion to form the resulting diagram
  \[
    \mc A^\otimes \xleftarrow{s} \mc K^\Gamma \xrightarrow{e} \Gamma \xrightarrow{q} \mc C^\otimes.
  \]
  We refer to $s$ as the \emph{source} map and to the composite $t\co \mc K^\Gamma \to \mc C^\otimes$ as the \emph{target} map, and refer to the diagram
  \[
    \mc A^\otimes \xleftarrow{s} \mc K^\Gamma \xrightarrow{t} \mc C^\otimes
  \]
  as the \emph{inert expansion of the span}.
\end{defn}

\begin{prop}
  \label{prop:expansionsumpreserving}
  Suppose that we have a sum-preserving span
  \[
    \mc A^\otimes \xleftarrow{p} \Gamma \xrightarrow{q} \mc C^\otimes
  \]
  with inert expansion
   \[
    \mc A^\otimes \xleftarrow{s} \mc K^\Gamma \xrightarrow{t} \mc C^\otimes.
  \]
  Suppose that $\oplus U_i \xrightarrow{\rho_i} U_i$ determine a sum diagram in $\mc A^\otimes$. Then any coCartesian lift of this diagram to $\mc K^\Gamma$ is mapped to a sum diagram in $\mc C^\otimes$.
\end{prop}

\begin{proof}
  Any coCartesian lift of this sum diagram to $\Inert(\mc A)$ is represented by a collection of diagrams of the form
  \[
    \begin{tikzcd}
      \oplus U_i \ar[r,"\rho_i"] \ar[d] & U_i \ar[d] \\
      \oplus \overline U_i \ar[r,"\overline \rho_i"] & \overline U_i, 
    \end{tikzcd}
  \]
  where the vertical arrows are determined by inert maps $U_i \to \overline U_i$. A further coCartesian lift to $\mc K^\Gamma$ adds diagrams
  \[
    X \rightarrow (\overline \rho_i)_! X
  \]
  for some $X \in \Gamma_{\oplus \overline U_i}$. This data is carried by $t$ to the diagrams
  \[
    q(X) \rightarrow q((\overline \rho_i)_! X)
  \]
  in $\mc C^\otimes$. By the assumption that the original span was sum-preserving, this is a sum diagram in $\mc C^\otimes$.
\end{proof}

\section{Operads of sections}

We will now give a brief description of the model structure for $\infty$-operads over a fixed base, using the formalism of categorical patterns \cite[Appendix B]{lurie-higheralgebra}.
\begin{defn}
  \label{def:naturalpattern}
  For an $\infty$-operad $\mc C^\otimes \to \NFin$, we define the \emph{natural} categorical pattern $\mf P_{\mc C} = (M_{\mc C},T_{\mc C},K_{\mc C})$ on $\mc C^\otimes$:
  \begin{itemize}
  \item $M_{\mc C}$ is the collection of inert edges in $\mc C^\otimes$,
  \item $T_{\mc C}$ is the collection of all 2-simplices in $\mc C^\otimes$, and
  \item $K_{\mc C}$ is the collection of all diagrams $\Lambda^2_0 \to \mc C^\otimes$, representing a pair of inert morphisms $X \to X_1$ and $X \to X_2$ that exhibit $X$ as a sum $X_1 \oplus X_2$ in $\mc C^\otimes$.\footnote{Arguably $K_{\mc C}$ should consist of all sum diagrams in $\mc C^\otimes$, rather than just binary sum diagrams. However, the proof goes through with or without this change, and the only difference for us is whether or not the empty category qualifies as fibrant in objects over $\mc C^\otimes$. (We believe that it should not; a fibrant object should be an $\infty$-operad, and the fiber over $\langle 0\rangle$ should be $\mc C^0$, a contractible category.)}
  \end{itemize}
\end{defn}

This categorical pattern gives rise to the structure of a Quillen model category on $\sSet^+/\mc C^\otimes$ \cite[B.0.20]{lurie-higheralgebra}. The fibrant objects in this ``natural'' model structure are fibrations $\mc D^\otimes \to \mc C^\otimes$ of $\infty$-operads with the natural marking on the source (cf. \cite[2.1.4.6]{lurie-higheralgebra}). 

\begin{thm}
  \label{thm:generalexpansion}
  Suppose that $\mc A^\otimes$ and $\mc C^\otimes$ are $\infty$-operads and that $\mc A^\otimes \xleftarrow{p} \Gamma \xrightarrow{q} \mc C^\otimes$ is a sum-preserving span of marked simplicial sets, with the marking on $\Gamma$ compatible with $p$.

  Then the functor
  \[
    X \mapsto X \times_{\mc A^\otimes} \mc K^\Gamma
  \]
  determines a left Quillen functor
  \[
    L\co \sSet^+/\mc A^\otimes \to \sSet^+/ \mc C^\otimes,
  \]
  where the model structures are determined by the natural categorical patterns. The right adjoint, applied to fibrant objects, determines a functor
  \[
    R\co \Op_\infty/ \mc C^\otimes \to \Op_\infty / \mc A^\otimes.
  \]
\end{thm}

\begin{proof}
  This proof will consist of showing that \cite[B.4.2]{lurie-higheralgebra} applies to the the inert expansion
  \[
    \mc A^\otimes \xleftarrow{s} \mc K^\Gamma \xrightarrow{r} \mc C^\otimes
  \]
  of this span. There are eight criteria to remember and verify, and we have already done the majority of this work.

  \begin{enumerate}
  \item The map $s$ is a flat categorical fibration.

    This is Corollary~\ref{cor:expansionflat}.

  \item The marked edges in $\mc K^\Gamma$ and the inert morphisms in $\mc A^\otimes$ are closed under equivalences and composition.

    This is Proposition~\ref{prop:markedclosure}.

  \item If $\sigma$ is a $2$-simplex of $\mc K^\Gamma$ such that $s(\sigma)$ is part of the categorical pattern on the source, then $r(\sigma)$ is part of the categorical pattern on the target.

    This is clear, because the natural categorical pattern includes all $2$-simplices.

  \item The pullback $s|_\gamma\co \mc K^\Gamma|_\gamma \to \Delta^1$ over any marked edge $\gamma\co \Delta^1 \to \mc A^\otimes$ is a Cartesian fibration.

    This is Proposition~\ref{prop:expansioncartesian}.
    
  \item Each of the simplicial sets $\Lambda^2_0$ in the source categorical pattern is an $\infty$-category, and the pullback $\mc K^\Gamma|_\lambda \to \Lambda^2_0$ over any limit diagram $\lambda\co \Lambda^2_0 \to \mc A^\otimes$ in the source pattern is a coCartesian fibration.

    The simplicial sets $\Lambda^2_0$ are $\infty$-categories. The rest is a consequence of Proposition~\ref{prop:expansioncocartesian}, because the map $s$ is a coCartesian fibration.

  \item Any coCartesian section of a limit diagram $\lambda$ in the source pattern maps, under $r$, to a limit diagram in the target pattern.

    This is Proposition~\ref{prop:expansionsumpreserving}.

  \item Suppose we are given a commutative diagram
    \[
      \begin{tikzcd}
        & Y \ar[dr, "g"]\\
        X \ar[ur, "f"] \ar[rr,"h"] && Z
      \end{tikzcd}
    \]
    in $\mc K^\Gamma$, where $g$ is locally $s$-Cartesian, $s(g)$ is inert, and $s(f)$ is an equivalence. Then $f$ is marked if and only if $h$ is marked.

    This is Proposition~\ref{prop:trianglecart}.

  \item Suppose we are given a commutative diagram
    \[
      \begin{tikzcd}
        & Y \ar[dr, "g"]\\
        X \ar[ur, "f"] \ar[rr,"h"] && Z
      \end{tikzcd}
    \]
    in $\mc K^\Gamma|_\lambda$ for some limit diagram $\lambda$ in the source pattern, where $f$ is $s|_\lambda$-coCartesian and $s|_\lambda(g)$ is an equivalence. Then $g$ is marked if and only if $h$ is marked.

    This is Proposition~\ref{prop:trianglecocart}.\qedhere
\end{enumerate}
\end{proof}

\begin{defn}
  \label{def:operadofsections}
  Under the assumptions of Theorem~\ref{thm:generalexpansion}, for a fibration of $\infty$-operads $D^\otimes \to \mc C^\otimes$ we define the $\infty$-operad of \emph{sections} to be
  \[
    \Sec_{\mc C^\otimes}(\Gamma, \mc D^\otimes)^\otimes = R(\mc D^\otimes). 
  \]
  We view this as a functor $\Op_\infty/\mc C^\otimes \to \Op_\infty/\mc A^\otimes$.
\end{defn}

In particular, if $\mc D^\otimes \to \mc C^\otimes$ is a fibration of $\infty$-operads, we get a fibration of $\infty$-operads $\Sec_{\mc C^\otimes}(\Gamma, \mc D^\otimes)^\otimes \to \mc A^\otimes$; we would like to describe its fibers more explicitly.

\begin{prop}
  \label{prop:generalexpansionfiber}
  For an object $U \in \mc A$, the fiber of the structure map
  \[
    \Sec_{\mc C^\otimes}(\Gamma, \mc D^\otimes)^\otimes \to \mc A^\otimes
  \]
  over $U$ is the category $\Fun^+_{\mc C^\otimes}(\Gamma_U, \mc D^\otimes)$ of marked lifts in the diagram
  \[
    \begin{tikzcd}
      & \mc D^\otimes \ar[d] \\
      \Gamma_U \ar[ur,dashed] \ar[r] & \mc C^\otimes.
    \end{tikzcd}
  \]
\end{prop}

\begin{proof}
  By adjunction, maps from $X$ to the fiber of $R(\mc D^\otimes)$ over $U$ are maps of marked simplicial sets
  \[
    X \times (\{U\} \times_{\mc A^\otimes} \mc K^\Gamma) \to \mc D^\otimes
  \]
  over $\mc C^\otimes$. The fiber of $\Inert(\mc A)$ over $U$ is equivalent to the two-object poset $(U \to U) \to (U \to 0)$, and hence the fiber product of $\{U\}$ with $\mc K^\Gamma$ is equivalent to the mapping cylinder of the diagram
\[
  \Gamma_U \xrightarrow{\rho_!} \Gamma_0
\]
where $\rho$ is the inert map $U \to 0$. A map of marked simplicial sets from the mapping cylinder of $X \times \Gamma_U \to X \times \Gamma_0$ to $\mc D^\otimes$ over $\mc C^\otimes$ is equivalent to a map $X \times \Gamma_U \to \mc D^\otimes$ over $\mc C^\otimes$, since $\Gamma_0$ must map to the fiber $\mc D^\otimes_{\langle 0\rangle}$ that is contractible (consisting entirely of terminal objects). As this is independent of $X$, this identifies the fiber with the desired section category.
\end{proof}

\section{Morphism spaces}

In this section our goal is to compute morphism spaces in these $\infty$-operads of sections.

\begin{prop}
  \label{prop:arrowliftcategory}
  Suppose that $\mc A^\otimes \xleftarrow{p} \Gamma \xrightarrow{q} \mc C^\otimes$ is a sum-preserving span of marked simplicial sets, with the marking on $\Gamma$ compatible with $p$, and that $\phi\co \Delta^1 \to \mc A^\otimes$ represents an active morphism $\oplus U_i \to V$. Then, for any map of $\infty$-operads $\mc D^\otimes \to \mc C^\otimes$, the functor category
  \[
    \Fun^\phi(\Delta^1, \Sec_{\mc C^\otimes}(\Gamma, \mc D^\otimes))
  \]
  of lifts of $\phi$ is equivalent to the limit of the following diagram:
  \[
    \begin{tikzcd}
      && \prod_i \Fun^+_{\mc C^\otimes}(\Gamma_{U_i},  \mc D^\otimes) \ar[d,"\oplus"]\\
      & \Fun^{+}_{T_\phi}(\Delta^1 \times \Gamma_{\oplus U_i},  \mc D^\otimes) \ar[d,"\ev_1"] \ar[r,"\ev_0"]
      & \Fun^+_{\oplus q|_{U_i}}(\Gamma_{\oplus U_i},  \mc D^\otimes) \\
      \Fun^+_{q|_V}(\Gamma_V,  \mc D^\otimes)\ar[r,"\phi_!"]
      &\Fun^+_{q|_V \circ \phi_!}(\Gamma_{\oplus U_i},  \mc D^\otimes)
    \end{tikzcd}
  \]
\end{prop}

\begin{proof}
  A map of simplicial sets $X \to \Fun^\phi(\Delta^1, \Sec_{\mc C^\otimes}(\Gamma, \mc D^\otimes))$ is, by definition, the data of a commutative diagram
  \[
    \begin{tikzcd}
      X \times \Delta^1 \ar[r] \ar[d] &
      \Sec_{\mc C^\otimes}(\Gamma, \mc D^\otimes) \ar[d] \\
      \Delta^1 \ar[r,"\phi",swap] &
      \mc A^\otimes
    \end{tikzcd}
  \]
  of marked simplicial sets, where $\Delta^1$ has the trivial marking. By adjunction, this is equivalent to a marked diagram
  \[
    \begin{tikzcd}
      X \times (\Delta^1 \times_{\mc A^\otimes} \mc K^\Gamma) \ar[r] \ar[d] &
      \mc D^\otimes \ar[d] \\
      \Delta^1 \times_{\mc A^\otimes} \mc K^\Gamma \ar[r,"\tilde \phi",swap] &
      \mc C^\otimes.
    \end{tikzcd}
  \]
  Therefore, we get an identification of functor categories
  \[
    \Fun^\phi(\Delta^1, \Sec_{\mc C^\otimes}(\Gamma, \mc D^\otimes)) \cong
    \Fun^+_{\mc C^\otimes} (\Delta^1 \times_{\mc A^\otimes} \mc K^\Gamma, \mc D^\otimes).
  \]
  As $\mc K^\Gamma \to \mc A^\otimes$ is a coCartesian fibration, the pullback over $\Delta^1$ is part of a homotopy pushout diagram
  \[
    \begin{tikzcd}
      \{1\} \times (\mc K^\Gamma)_{\oplus U_1} \ar[r] \ar[d] &
      \Delta^1 \times (\mc K^\Gamma)_{\oplus U_1} \ar[d] \\
      (\mc K^\Gamma)_{V} \ar[r] &
      \Delta^1 \times_{\mc A^\otimes} \mc K^\Gamma.
    \end{tikzcd}
  \]
  Inside the pushout is a smaller simplicial set which is the colimit of the diagram
  \[
    \begin{tikzcd}
      &\{0\} \times \Gamma_{\oplus U_i} \ar[r] \ar[d] &
      \{0\} \times (\mc K^\Gamma)_{\oplus U_1}\\
      \{1\} \times \Gamma_{\oplus U_1} \ar[d,"\phi_!"] \ar[r] &
      \Delta^1 \times \Gamma_{\oplus U_i} \\
      \Gamma_{V}.
    \end{tikzcd}
  \]
  The only simplices not in this smaller simplicial set have final vertex that maps to the fiber $\mc C^\otimes_{\langle 0\rangle}$. Because the fibers $\mc C^\otimes_{\langle 0\rangle}$ and $\mc D^\otimes_{\langle 0\rangle}$ are contractible, consisting entirely of terminal objects, maps \emph{to} $\mc C^\otimes$ or $\mc D^\otimes$ from the smaller simplicial set are equivalent to maps to them from the larger.

  Applying $\Fun^+_{\mc C^\otimes}(-,\mc D^\otimes)$ to the above colimit diagram then produces the desired limit diagram.
\end{proof}

For convenience, given functors $F_i\co \Gamma_{U_i} \to \mc D^\otimes$ over $\mc C^\otimes$, we will write $\vec F$ for the composite
\[
  \Gamma_{\oplus U_i} \to \prod \Gamma_{U_i} \to \prod \mc D^\otimes \xrightarrow{\oplus} \mc D^\otimes.
\]
Recall from Remark~\ref{rmk:spanfunctor} that there is an associated natural transformation $T_\phi\co (q|_{\oplus U_i}) \Rightarrow (q|_V) \circ \phi_!$ between functors $\Gamma_{\oplus U_i} \to \mc C^\otimes$.

\begin{cor}
  \label{cor:morphismspaces}
  Suppose that $F_i\co \Gamma_{U_i} \to \mc D^\otimes$ and $G\co \Gamma_V \to \mc D^\otimes$ are objects of $\Sec_{\mc C^\otimes}(\Gamma, \mc D^\otimes)$. Then the space
  \[
    \Mul^\phi_{\Sec_{\mc C^\otimes}(\Gamma, \mc D^\otimes)}(\{F_i\}_{1 \leq i \leq n}, G),
  \]
  parametrizing lifts of the active morphism $\phi$, is equivalent to the space
  \[
    \Nat^{T_\phi}(\vec F, G \circ \phi_!)
  \]
  of lifts of $T_\phi$ to a natural transformation of marked functors.
\end{cor}

\begin{proof}
  This follows from Proposition~\ref{prop:arrowliftcategory} by taking the fiber of the functor category $\Fun^\phi(\Delta^1, \Sec_{\mc C^\otimes}(\Gamma, \mc D^\otimes))$ over $(F_i) \in \prod_i \Fun^+_{\mc C^\otimes}(\Gamma_{U_i},  \mc D^\otimes)$ and $G \in \Fun^+_{\mc C^\otimes}(\Gamma_V,  \mc D^\otimes)$.
\end{proof}

\section{CoCartesian properties}

\begin{defn}
  \label{def:objectwisenatural}
  Suppose that $r\co \mc D \to \mc C$ is a coCartesian fibration and that $T\co \Delta^1 \times L \to \mc C$ represents a natural transformation $T\co f \Rightarrow g$ between functors $f,g\co L \to \mc C$. Then we define
  \[
    T_!\co \Fun_f(L,\mc D) \to \Fun_g(L,\mc D)
  \]
  to be the map on functor categories that, object-by-object, applies the functor associated to the natural transformation:
  \[
    (T_! F)(x) = (Tx)_! (F(x)).
  \]
\end{defn}

The following records some basic properties of this construction.

\begin{prop}
  \label{prop:objectwisecomposition}
  For composable natural transformations $T\co f \Rightarrow g$ and $S\co g \Rightarrow h$, we have a natural equivalence
  \[
    (S\cdot T)_!(F) \simeq S_! (T_! (F)).
  \]
  For a functor $\alpha\co L \to L'$ and a natural transformation $T\co f \Rightarrow g$ between functors $L' \to \mc C$, we have a natural equivalence
  \[
    (T_! F) \circ \alpha \simeq (T \circ \alpha)_! (F \circ \alpha).
  \]
  For functors $F\co L \to \mc D$ over $f$ and $G\co L \to \mc D$ over $g$ and a natural transformation $T\co f \Rightarrow g$, there is an equivalence between spaces of lifted natural transformations
  \[
    \Nat^T(F,G) \simeq \Nat^g(T_! F, G).
  \]
\end{prop}




\begin{prop}
  \label{prop:morphismspacescocartesian}
  Suppose that $\mc D^\otimes \to \mc C^\otimes$ is a coCartesian fibration, and that $F_i\co \Gamma_{U_i} \to \mc D^\otimes$ and $G\co \Gamma_V \to \mc D^\otimes$ are objects of $\Sec_{\mc C^\otimes}(\Gamma, \mc D^\otimes)$. Then the space
  \[
    \Mul^\phi_{\Sec_{\mc C^\otimes}(\Gamma, \mc D^\otimes)}(\{F_i\}_{1 \leq i \leq n}, G),
  \]
  parametrizing lifts of an active morphism $\phi\co \oplus U_i \to V$, is equivalent to the space
  \[
    \Nat^{q|_V \circ \phi_!}((T_\phi)_! (\vec F), G \circ \phi_!)
  \]
  of natural transformation of marked functors $\Gamma_{\oplus U_i} \to \mc D^\otimes$ over $q|_V \circ \phi_!\co \Gamma_{\oplus U_i} \to \mc C^\otimes$.
\end{prop}

\begin{proof}
  This is an application of Corollary~\ref{cor:morphismspaces}.
\end{proof}

\begin{prop}
  \label{prop:prodprescocartesian}
  Suppose that $\mc A^\otimes \xleftarrow{p} \Gamma \xrightarrow{q} \mc C^\otimes$ is a sum-preserving span of marked simplicial sets, with the marking on $\Gamma$ compatible with $p$. Fix a coCartesian fibration of $\infty$-operads $\mc D^\otimes \to \mc C^\otimes$.

  For an active morphism $\phi\co \oplus U_i \to V$ in $\mc A^\otimes$, if the induced map
  \[
    \phi_!\co \Gamma_{\oplus U_i} \to \Gamma_V
  \]
  is an equivalence, then the structure map
  \[
    s\co \Sec_{\mc C^\otimes}(\Gamma, \mc D^\otimes)^\otimes \to \mc A^\otimes
  \]
  has $s$-coCartesian lifts of $\phi$. 
\end{prop}

\begin{proof}
  By a standard reduction decomposing general objects into sums, it suffices to construct a map $\eta\co \oplus F_i \to F$ over $\phi$ such that for any $\beta\co V \to W$ in $\mc A$, composition with $\eta$ determines an equivalence
  \[
    \Map^\beta_{\Sec_{\mc C^\otimes}(\Gamma, \mc D^\otimes)}(F,G) \simeq
    \Mul^{\beta\phi}_{\Sec_{\mc C^\otimes}(\Gamma, \mc D^\otimes)}(\{F_i\}, G).
  \]
  However, applying Proposition~\ref{prop:morphismspacescocartesian} and Proposition~\ref{prop:objectwisecomposition} to the identities from Remark~\ref{rmk:spanfunctor} we have a sequence of natural equivalences:
  \begin{align*}
    \Mul^{\beta\phi}(\{F_i\},G)
    &\simeq
      \Nat^{q|_W \circ (\beta\phi)_!}((T_{\beta\phi})_! \vec F, G (\beta\phi)_!)\\
    &\simeq
      \Nat^{q|_W \circ \beta_!\phi_!}((T_{\beta} \circ \phi_!)_! (T_\phi)_!\vec F, G \beta_! \phi_!)\\
    &\simeq 
      \Nat^{q|_W \circ \beta_!}((T_{\beta} \circ \phi_!)_! (T_\phi)_!(\vec F) \circ (\phi_!)^{-1}, G \beta_!)\\
    &\simeq 
      \Nat^{q|_W \circ \beta_!}((T_{\beta})_! (T_\phi)_!(\vec F \circ (\phi_!)^{-1}), G \beta_!)\\
    &\simeq 
      \Map^{\beta}((T_\phi)_!(\vec F \circ (\phi_!)^{-1}), G).
  \end{align*}

  This determines an object
  \[
    F = (T_\phi)_!(\vec F (\phi_!)^{-1})
  \]
  with an $s$-coCartesian map $\eta\co \oplus F_i \to F$ in ${\Sec_{\mc C^\otimes}(\Gamma, \mc D^\otimes)}$, as desired.
\end{proof}

\begin{rmk}
  The explicit target $\phi_!(F_1,\dots,F_n) = (T_\phi)_!(\vec F \circ (\phi_!)^{-1})$ of these coCartesian lifts allows us to calculate the values of this functor. If $\phi_!(X_1,\dots,X_n) = Y$, then
  \[
    \phi_!(F_1,\dots,F_n) (Y) = (T_\phi(X_1,\dots,X_n))_! (F_1(X_1),\dots,F_n(X_n)).
  \]
\end{rmk}

\section{The Boardman--Vogt tensor}

Two $\infty$-operads $\mc A^\otimes$ and $\mc E^\otimes$ have a Boardman--Vogt tensor product $(\mc A \bvt \mc E)^\otimes$, whose defining property is that it is universal among $\infty$-operads with a pairing of marked simplicial sets
\[
  \mc A^\otimes \times \mc E^\otimes \to (\mc A \bvt \mc E)^\otimes
\]
over the smash product functor $\NFin \times \NFin \to \NFin$ \cite[2.2.5]{lurie-higheralgebra}.

The pairing preserves sum diagrams in each variable separately. As a result, the Boardman--Vogt tensor can be a natural source of sum-preserving spans.

\begin{defn}
  \label{def:BVtwistable}
  A diagram $\mc A^\otimes \xleftarrow{p} \Gamma \xrightarrow{r} \mc E^\otimes$ is \emph{sum-trivializing} if $p$ is a coCartesian fibration and $r\co \Gamma \to \mc E^\otimes$ takes $p$-coCartesian maps in $\Gamma$ to equivalences. 
\end{defn}

\begin{prop}
  \label{prop:BVtwist}
  Suppose that we have a sum-trivializing diagram $\mc A^\otimes \xleftarrow{p} \Gamma \xrightarrow{r} \mc E^\otimes$. Then the composite map
  \[
    q\co \Gamma \xrightarrow{(p,r)} \mc A^\otimes \times \mc E^\otimes \to (\mc A \bvt \mc E)^\otimes,
  \]
  makes the span $\mc A^\otimes \leftarrow \Gamma \to (\mc A \bvt \mc E)^\otimes$ sum-preserving.
\end{prop}

\begin{proof}
  A $p$-coCartesian lift of a sum diagram is taken by $p$ to a sum diagram, and by $r$ to an essentially constant diagram. Therefore, the image in $(\mc A \bvt \mc E)^\otimes$ becomes a sum diagram.
\end{proof}

\begin{defn}
Fix a fibration $\tau\co \mc V^\otimes \to (\mc A \bvt \mc E)^\otimes$ of $\infty$-operads. For any object $X \in \mc A$, the composite
\[
  \{X\} \times \mc E^\otimes \into \mc A^\otimes \times \mc E^\otimes \to (\mc A \bvt \mc E)^\otimes
\]
is a map of $\infty$-operads. We write $\mc V^\otimes_X \to \mc E^\otimes$ for the pullback fibration of $\infty$-operads.
\end{defn}

\begin{rmk}
  Suppose $\tau$ is further a coCartesian fibration giving $\mc V$ the structure of an $(\mc A \bvt \mc E)$-monoidal $\infty$-category. Then its restriction makes $\mc V_X$ into a $\mc E$-monoidal $\infty$-category, and any $f\co X \to Y$ in $\mc A$ then induces a lax $\mc E$-monoidal functor $f_!\co \mc V_X \to \mc V_Y$.
  
  More generally, an active morphism $f\co \oplus X_i \to Y$ in $\mc A^\otimes$ induces a functor
  \[
    f_!\co \mc V_{X_1}^\otimes \times_{\mc E^\otimes} \dots \times_{\mc E^\otimes} \mc V_{X_n}^\otimes \to \mc V_Y^\otimes
  \]
  over $\mc E^\otimes$.
\end{rmk}

\begin{prop}
  \label{prop:BVexpansionfiber}
  Suppose that we have a sum-trivializing diagram $\mc A^\otimes \xleftarrow{p} \Gamma \xrightarrow{r} \mc E^\otimes$ of marked simplicial sets, with the marking on $\Gamma$ compatible with $p$. Then there is a Quillen adjunction
  \[
    \sSet^+/(\mc A \bvt \mc E)^\otimes \leftrightarrows \sSet^+/\mc A^\otimes,
  \]
  inducing a functor
  \[
    \Fun^+_{\mc E^\otimes}(\Gamma,-)^\otimes \co \Op_\infty / (\mc A \bvt \mc E)^\otimes \to \Op_\infty / \mc A^\otimes.
  \]
  For any fibration of $\infty$-operads $\mc V^\otimes \to (\mc A \bvt \mc E)^\otimes$ and $X$ in $\mc A$, the fiber of $\Fun_{\mc E^\otimes}(\Gamma,\mc V^\otimes)^\otimes \to \mc A^\otimes$ over $X$ is the $\infty$-category
  \[
    \Fun^+_{\mc E^\otimes}(\Gamma, \mc V^\otimes)^\otimes_X \simeq \Fun^+_{\mc E^\otimes}(\Gamma_X, \mc V^\otimes_X)
  \]
  of marked lifts in the diagram
  \[
    \begin{tikzcd}
      & \mc V^\otimes_X \ar[d] \\
      \Gamma_X \ar[ur,dashed] \ar[r] & \mc E^\otimes.
    \end{tikzcd}
  \]
\end{prop}

\begin{proof}
  We first note that the composite map
  \[
    \Gamma \xrightarrow{q} \mc A^\otimes \times \mc E^\otimes \to (\mc A \bvt \mc E)^\otimes
  \]
  is also a map of marked simplicial sets, because the individual terms are, and it is sum-preserving by Proposition~\ref{prop:BVtwist}. We can then apply Theorem~\ref{thm:generalexpansion}.

  Proposition~\ref{prop:generalexpansionfiber} identifies the fiber with the $\infty$-category of marked lifts in the diagram
  \[
    \begin{tikzcd}
      && \mc V^\otimes \ar[d] \\
      \Gamma_X \ar[urr,dashed] \ar[r] & \{X\} \times \mc E^\otimes \ar[r]
      & (\mc A \bvt \mc E)^\otimes.
    \end{tikzcd}
  \]
  However, these are the same as marked lifts to the pullback, which is $\mc V^\otimes_X$ by definition.
\end{proof}

\section{Fiber product spans}

The following applies Proposition~\ref{prop:BVexpansionfiber} to a special case that is, despite appearances, more straightforward to verify.
\begin{thm}
  \label{thm:mainspan}
  Suppose that we have a diagram
  \[
    \mc O^\otimes \to \mc P^\otimes \leftarrow \mc A^\otimes \leftarrow \Gamma \to D \to \mc E^\otimes
  \]
  of marked simplicial sets such that:
  \begin{itemize}
    \item the objects $\mc A^\otimes$, $\mc E^\otimes$, $\mc O^\otimes$, and $\mc P^\otimes$ are $\infty$-operads with their natural markings;
    \item the maps $\Gamma \to \mc A^\otimes \to \mc P^\otimes$ are both coCartesian fibrations; and
    \item the diagram $\mc A^\otimes \leftarrow \Gamma \to D$ represents a $\mc P$-monoidal functor $\mc A \to \sSet^+/D$, where the latter is symmetric monoidal under fiber product.
  \end{itemize}
  Then there is a Quillen adjunction
  \[
    \sSet^+/(\mc O \bvt \mc E)^\otimes \leftrightarrows \sSet^+/\mc O^\otimes,
  \]
  whose right adjoint induces a functor
  \[
    \Fun^+_D(\Gamma,-)^\otimes\co \Op_\infty/(\mc O \bvt \mc E)^\otimes \to \Op_\infty/\mc O^\otimes,
  \]
  with the following properties.
  \begin{itemize}
  \item For any $U \in \mc P$ and any $X \in \mc O$ over $U$, the fiber $\Fun^+_D(\Gamma,\mc V^\otimes)^\otimes_X$ over $X$ is the $\infty$-category of pairs $(Y,s)$ of an object $Y \in \mc A_U$ and a marked lift in the diagram
    \[
      \begin{tikzcd}
        \Gamma_Y \ar[r,dashed,"s"] \ar[d] & \mc V^\otimes_X \ar[d] \\
        D \ar[r] & \mc E^\otimes.
      \end{tikzcd}
    \]
    More specifically, a map $(Y,s) \to (Z,t)$ consists of a map $f\co Y \to Z$ in $\mc A_U$ and a natural transformation $s \Rightarrow t f_!$ over $\mc E^\otimes$.
  \item The functor $\Fun^+_D(\Gamma,-)^\otimes$ preserves coCartesian fibrations: if $\mc V$ is $(\mc O \bvt \mc E)$-monoidal, then $\Fun^+_D(\Gamma,\mc V^\otimes)^\otimes$ is $\mc O$-monoidal.
  \item The $\mc O$-monoidal structure is expressed as follows. Suppose we have any active morphism $\phi\co \oplus X_i \to Z$ in $\mc O^\otimes$ lying over $\overline \phi\co \oplus U_i \to V$ in $\mc P^\otimes$. Then, for any objects $(Y_i, s_i) \in \Fun^+_D(\Gamma, \mc V^\otimes)^\otimes_{X_i}$ as above, the object $\phi_!((Y_1,s_1),\dots,(Y_n,s_n))$ is represented by the composite
    \[
      \Gamma_{{\overline \phi}_!(Y_1,\dots,Y_n)}
      \xrightarrow{\sim} \Gamma_{Y_1} \times_D \dots \times_D \Gamma_{Y_n}
      \xrightarrow{\prod s_i} \mc V^\otimes_{X_i} \times_{\mc E^\otimes} \dots \times_{\mc E^\otimes} \mc V^\otimes_{X_i}
      \xrightarrow{\phi_!} \mc V^\otimes_Y.
    \]
  \end{itemize}
\end{thm}

\begin{proof}
  We first define $\tilde \Gamma = \mc O^\otimes \times_{\mc P^\otimes} \Gamma$, and note that there is an associated natural diagram
  \[
    \mc O^\otimes \times_{\mc P^\otimes} \mc A^\otimes \xleftarrow{p} \tilde \Gamma \xrightarrow{q} \mc E^\otimes
  \]
  of marked simplicial sets. The map $p$ is a coCartesian fibration; it is the base-change of the map $\Gamma \to \mc A^\otimes$, the $p$-coCartesian edges map to coCartesian edges in $\Gamma$, and there is a canonical isomorphism $\tilde\Gamma_{(X,Y)} \cong \Gamma_X$ between fibers. The $p$-coCartesian edges map to equivalences in $D$ because all coCartesian edges in $\sSet^+/D$ map to equivalences in $D$, and so this diagram is sum-trivializing. By Proposition~\ref{prop:BVtwist}, the associated diagram
  \[
    \mc O^\otimes \times_{\mc P^\otimes} \mc A^\otimes \xleftarrow{p} \mc O^\otimes \times_{\mc P^\otimes} \Gamma \to ((\mc O \times_{\mc P} \mc A) \bvt \mc E)^\otimes \to (\mc O \bvt \mc E)^\otimes .
  \]
  is sum-preserving. By Proposition~\ref{prop:BVexpansionfiber}, we get a Quillen adjunction inducing a functor
  \[
    \Op_\infty/(\mc O \bvt \mc E)^\otimes \to \Op_\infty/\mc O^\otimes \times_{\mc P^\otimes} \mc A^\otimes.
  \]
  that we denote by $\Fun^+_D(\Gamma,-)^\otimes$. Further, for any fibration $\mc V^\otimes \to (\mc O \bvt \mc E)^\otimes$, the fiber over $(X,Y) \in \mc O \times_{\mc P} \mc A$ is the category of marked lifts in the diagram
  \[
    \begin{tikzcd}
      \Gamma_X \ar[d] \ar[rr,dashed]&&
      \mc V^\otimes_Y \ar[d]\\
      \{(X,Y)\} \times D \ar[r] &
      \{Y\} \times \mc E^\otimes \ar[r] &
      (\mc O \bvt \mc E)^\otimes.
    \end{tikzcd}
  \]

  We now wish to show that the composite
  \[
    \Fun^+_D(\Gamma, \mc V^\otimes)^\otimes \to \mc O^\otimes \times_{\mc P^\otimes} \mc A^\otimes \to \mc O^\otimes
  \]
  is a coCartesian fibration if the map $\mc V^\otimes \to (\mc O \bvt \mc E)^\otimes$ was.

  By the standard reduction, it suffices to show that active maps have coCartesian lifts. Suppose that we have an active map $f\co \oplus X_i \to X$ in $\mc O^\otimes$ with image $\overline f\co \oplus U_i \to V$ in $\mc P^\otimes$, together with lifts $(X_i, Y_i, s_i)$ to $\Fun^+_D(\Gamma, \mc V^\otimes)$. Because $\mc A^\otimes \to \mc P^\otimes$ is a coCartesian fibration, there exists a coCartesian lift $\tilde f\co \oplus Y_i \to Y$ of $\overline f$ to $\mc A^\otimes$, and the corresponding map $(f,\tilde f)\co \oplus (X_i,Y_i) \to (X,Y)$ is a coCartesian lift of $f$ to $\mc O^\otimes \times_{\mc P^\otimes} \mc A^\otimes$.

  By assumption, $\Gamma$ represents a $\mc P$-monoidal functor from $\mc A^\otimes$ to $\sSet^+/D$, so the induced functor
  \[
    (\tilde f)_!\co \Gamma_{\oplus Y_i} \to \Gamma_Y
  \]
  is an equivalence. This makes the map
  \[
    \tilde \Gamma_{\oplus (X_i, Y_i)} \to \tilde \Gamma_{(X,Y)}
  \]
  also an equivalence. Therefore, by Proposition~\ref{prop:prodprescocartesian} the map $(f,\tilde f)$ has a coCartesian lift from $\mc O^\otimes \times_{\mc P^\otimes} \mc A^\otimes$ to $\Fun^+_D(\Gamma,\mc V^\otimes)^\otimes$, represented by the composite
  \[
    \Gamma_{Y}
      \xrightarrow{\sim} \Gamma_{Y_1} \times_D \dots \times_D \Gamma_{Y_n}
      \xrightarrow{\prod s_i} \mc V^\otimes_{X_1} \times_{\mc E^\otimes} \dots \times_{\mc E^\otimes} \mc V^\otimes_{X_n}
      \xrightarrow{\phi_!} \mc V^\otimes_Y.
  \]
  This edge is then also a coCartesian lift of $f$ from $\mc O^\otimes$ to $\Fun^+_D(\Gamma,\mc V^\otimes)^\otimes$, as desired.
\end{proof}

The most straightforward case is when $\mc P$ is the commutative $\infty$-operad, which we can record now.
\begin{cor}
  \label{cor:mainspan}
  Suppose that we have a diagram
  \[
    \mc A^\otimes \leftarrow \Gamma \to D \to \mc E^\otimes
  \]
  of marked simplicial sets such that:
  \begin{itemize}
  \item the objects $\mc A^\otimes$ and $\mc E^\otimes$ are $\infty$-operads with their natural markings;
  \item the map $\Gamma \to \mc A^\otimes$ is a coCartesian fibration;
  \item the diagram $\mc A^\otimes \leftarrow \Gamma \to D$ represents a symmetric monoidal functor $\mc A \to \sSet^+/D$, where the latter is symmetric monoidal under fiber product.
  \end{itemize}
  Then, for any $\infty$-operad $\mc O^\otimes$, there is a Quillen adjunction
  \[
    \sSet^+/(\mc O \bvt \mc E)^\otimes \leftrightarrows \sSet^+/\mc O^\otimes,
  \]
  whose right adjoint induces a functor
  \[
    \Op_\infty/(\mc O \bvt \mc E)^\otimes \to \Op_\infty/\mc O^\otimes,
  \]
  written as $\mc V^\otimes \mapsto \Fun^+_D(\Gamma,\mc V^\otimes)^\otimes$,
  with the following properties.
  \begin{itemize}
  \item For any $X \in \mc O$, the fiber $\Fun^+_D(\Gamma,\mc V^\otimes)^\otimes_X$ over $X$ is the $\infty$-category of pairs $(Y,s)$ of an object $Y \in \mc A$ and a marked lift in the diagram
    \[
      \begin{tikzcd}
        \Gamma_Y \ar[r,dashed,"s"] \ar[d] & \mc V^\otimes_X \ar[d] \\
        D \ar[r] & \mc E^\otimes.
      \end{tikzcd}
    \]
    More specifically, a map $(Y,s) \to (Z,t)$ consists of a map $f\co Y \to Z$ in $\mc A$ and a natural transformation $s \Rightarrow t f_!$ over $\mc E^\otimes$.
  \item The functor $\Fun^+_D(\Gamma,-)^\otimes$ preserves coCartesian fibrations: if $\mc V$ is $(\mc O \bvt \mc E)$-monoidal, then $\Fun^+_D(\Gamma,\mc V^\otimes)^\otimes$ is $\mc O$-monoidal.
  \item The $\mc O$-monoidal structure is expressed as follows. Suppose we have any active morphism $\phi\co \oplus X_i \to Z$ in $\mc O^\otimes$. Then, for any objects $(Y_i, s_i) \in \Fun^+_D(\Gamma, \mc V^\otimes)^\otimes_{X_i}$ as above, the object $\phi_!((Y_1,s_1),\dots,(Y_n,s_n))$ is represented by the composite
    \[
      \Gamma_{{\overline \phi}_!(Y_1,\dots,Y_n)}
      \xrightarrow{\sim} \Gamma_{Y_1} \times_D \dots \times_D \Gamma_{Y_n}
      \xrightarrow{\prod s_i} \mc V^\otimes_{X_1} \times_{\mc E^\otimes} \dots \times_{\mc E^\otimes} \mc V^\otimes_{X_n}
      \xrightarrow{\phi_!} \mc V^\otimes_Y.
    \]
  \end{itemize}
\end{cor}

\section{Enrichments}

The following setup is based on Gepner--Haugseng's approach to enriched categories from \cite[4.1]{gepner-haugseng-enriched}.
\begin{defn}
  \label{def:setx}
  Let $\Set^\times$ denote the category of pairs $(S,(X_s)_{s \in S})$ of a set $S$ and an $S$-indexed tuple of sets, associated to the Cartesian symmetric monoidal structure on sets \cite[2.1.1.7]{lurie-higheralgebra}.
\end{defn}

\begin{defn}
  \label{def:setxdiagram}
  The functor $\Delta^\op_{(-)}\co \Set^\times \times \Delta^\op \to \Set$ takes a tuple $(S,(X_s), \underbar n)$ of a finite set $S$, an $S$-indexed tuple $(X_s)$, and a finite ordered set $\underbar n = \{0 < 1 < \dots < n\}$, and sends it to the set
  \[
    \prod_{s \in S} \Map(\underbar n, X_s) \cong \prod_{s \in S} X^{n+1}.
  \]
\end{defn}

\begin{defn}
  \label{def:enrichmentdiagram}
  Let $\Lambda \to \Set^\times \times \Delta^\op$ denote the associated Grothendieck wreath product of $f$.
\end{defn}

The objects of $\Lambda$ are tuples
\[
  (S,\underbar n, (X_s)_{s \in S},(\phi_s)_{s \in S})
\]
where $\phi_s\co \underbar n \to X_s$ is a map of sets representing an element of $X^{n+1}$. The fiber over $(S, (X_s))$ is the category $\Delta^\op_{\prod X_s}$ of \cite[4.1.1]{gepner-haugseng-enriched}.

\begin{defn}
  \label{def:enrichmentspan}
  The \emph{enrichment span} is the diagram of maps
  \[
    \NFin \leftarrow N\Set^\times \xleftarrow{p} N\Lambda \xrightarrow{r}  N\Delta^\op \to \mb E_1^\otimes
  \]
  viewed as a diagram of $\infty$-categories.
\end{defn}

\begin{prop}
  \label{prop:enrichmentspanfibration}
  In the span
  \[
    N\Set^\times \xleftarrow{p} N\Lambda \xrightarrow{r} N\Delta^\op
  \]
  the maps $p$ and $r$ are coCartesian fibrations. This represents a product-preserving functor
  \[
    X \mapsto \Delta^\op_X\co N\Set^\times \to \sSet^+ / N\Delta^\op.
  \]
  from sets to marked simplicial sets over $N\Delta^\op$.
\end{prop}

\begin{proof}
  The first is automatically true of the span $\mc M \xleftarrow{p} \Gamma \xrightarrow{r} \mc N$ associated to any two-variable functor $\mc M \times \mc N \to \Cat_\infty$. The fact that this functor is product-preserving is clear from the explicit description of $\Delta^\op_X$.
\end{proof}

For any $\infty$-operad $\mc O^\otimes \to \NFin$, specializing Corollary~\ref{cor:mainspan} to the diagram
\[
  N\Set^\times \xleftarrow{p} N\Lambda \xrightarrow{r}  N\Delta^\op \to \mb E_1^\otimes
\]
gives the following result.
\begin{thm}
  \label{thm:thmmainmain}
  For any $\infty$-operad $\mc O^\otimes$, there is a Quillen adjunction
  \[
    \sSet^+/(\mc O \bvt \mb E_1)^\otimes \leftrightarrows \sSet^+/\mc O^\otimes,
  \]
  whose right adjoint induces a functor
  \[
    \Cat_\infty^{(-)}\co \Op_\infty / (\mc O \bvt \mb E_1)^\otimes \to \Op_\infty / \mc O^\otimes
  \]
  with the following properties.
  \begin{itemize}
  \item For any $X \in \mc O$, the fiber $(\Cat_\infty^{\mc V})_X$ over $X$ is the $\infty$-category $\Cat_\infty^{\mc V_{X}}$ of $\mc V^\otimes_X$-enriched categories: pairs $(Y,s)$ of a set $Y$ and a diagram
    \[
      \begin{tikzcd}
        \Delta^\op_Y \ar[r,dashed,"s"] \ar[d] &
        \mc V^\otimes_X \ar[d] \\
        N\Delta^\op \ar[r] & \mb E_1^\otimes
      \end{tikzcd}
    \]
    of marked simplicial sets.
    
    A map $(Y,s) \to (Z,t)$ in $\Cat_\infty^{\mc V}$ consists of a map $g\co Y \to Z$ of sets and a natural transformation $s \Rightarrow t g_!$ over $\mb  E_1^\otimes$.
  \item If $\mc V$ is $(\mc O \bvt \mb E_1)$-monoidal, then this makes $\Cat_\infty^{\mc V}$ an $\mc O$-monoidal $\infty$-category.
  \item The $\mc O$-monoidal structure is expressed as follows. Suppose we have any active morphism $\phi\co \oplus X_i \to Z$ in $\mc O^\otimes$. Then, for any objects $(Y_i, s_i) \in \Cat_\infty^{\mc V_{X_i}}$ as above, the object $\phi_!((Y_1,s_1),\dots,(Y_n,s_n))$ is represented by the pair $(\prod Y_i, s)$, where $s$ is the composite
    \[
      \Delta^\op_{\prod Y_i} 
      \simeq \Delta^\op_{Y_1} \times_{\Delta^\op} \dots
      \times_{\Delta^\op} \Delta^\op_{Y_n}
      \xrightarrow{\prod s_i}
      \mc V^\otimes_{X_1} \times_{\mb E_1^\otimes} \dots \times_{\mb E_1^\otimes} \mc V^\otimes_{X_n}
      \xrightarrow{\phi_!} \mc V^\otimes_Y.
    \]
  \end{itemize}
\end{thm}

This implies Theorem~\ref{thm:mainthm}.

\nocite{kelly-enriched}

\bibliography{../masterbib}
\end{document}